\documentclass[11pt, letterpaper, english]{amsart}
\usepackage{amsmath}
\usepackage{amsthm}
\usepackage{amssymb}
\usepackage[dvipsnames]{xcolor}
\usepackage{amscd}
\usepackage{mathtools}
\usepackage{booktabs}
\usepackage[boxed]{algorithm2e}
\usepackage{subcaption}
\usepackage{array}
\usepackage[cal=boondoxo,scr=euler]{mathalfa}
\usepackage[backref=page,linktocpage]{hyperref}
\usepackage{cleveref}
\usepackage{caption}
\usepackage{graphics,graphicx}
\usepackage{tikz,tikz-cd}
\usetikzlibrary{graphs,graphs.standard,calc,decorations.pathreplacing,automata}
\usepackage{enumerate}
\DeclareMathAlphabet{\mathsf}{OT1}{\sfdefault}{m}{n}
\newcommand{\nocontentsline}[3]{}
\newcommand{\tocless}[2]{\bgroup\let\addcontentsline=\nocontentsline#1{#2}\egroup}
\usepackage[margin=1.5in]{geometry}
\usepackage{verbatim}
\usepackage{scalerel}
\makeatletter
\def\dual#1{\expandafter\dual@aux#1\@nil}
\def\dual@aux#1/#2\@nil{\begin{tabular}{@{}c@{}}#1\\#2\end{tabular}}
\@namedef{subjclassname@2020}{\textup{2020} Mathematics Subject Classification}
\makeatother

\DeclareMathAlphabet{\amathbb}{U}{bbold}{m}{n}

\hypersetup{colorlinks=true,linkbordercolor=white,linkcolor=NavyBlue,
 anchorcolor=black,citecolor=NavyBlue,filecolor=cyan,menucolor=NavyBlue,
 runcolor=cyan,urlcolor=NavyBlue,
 pdftitle={Unimodality shenanigans in Ehrhart theory},
 pdfauthor={Luis Ferroni}}
\newtheoremstyle{teoremas}{12pt}{13pt}{\itshape}{}{\bfseries}{}{.5em}{}
\theoremstyle{teoremas}
\newtheorem{theorem}{Theorem}[section]
\newtheorem{corollary}[theorem]{Corollary}

\newtheorem{proposition}[theorem]{Proposition}
\newcommand{\repeatedtheoremname}{}
\newtheorem*{repeatedtheorem}{\repeatedtheoremname}
\newenvironment{restate}[1]{%
  \renewcommand{\repeatedtheoremname}{Theorem~\hyperref[#1]{\ref*{#1}}}%
  \begin{repeatedtheorem}%
}{\end{repeatedtheorem}}
\newtheoremstyle{definition}{12pt}{12pt}{}{}{\bfseries}{}{.5em}{}
\theoremstyle{definition}

\newtheorem{conjecture}[theorem]{Conjecture}

\crefname{theorem}{theorem}{theorems}
\Crefname{theorem}{Theorem}{Theorems}
\crefname{lemma}{lemma}{lemmas}
\Crefname{lemma}{Lemma}{Lemmas}
\crefname{proposition}{proposition}{propositions}
\Crefname{proposition}{Proposition}{Propositions}
\crefname{corollary}{corollary}{corollaries}
\Crefname{corollary}{Corollary}{Corollaries}

\newcommand{\R}{\mathbb{R}}
\newcommand{\Z}{\mathbb{Z}}

\DeclareMathOperator*{\conv}{conv}

\newcommand{\cC}{\mathcal C}
\AtBeginDocument{\def\MR#1{}}

\title[Unimodality shenanigans in Ehrhart theory]{Unimodality shenanigans in Ehrhart theory}
\author[L.~Ferroni]{Luis Ferroni}
\address{(L. Ferroni) Dipartimento di Matematica, Universit\`a di Pisa, Pisa, Italy.}
\email{luis.ferroni@unipi.it}
\keywords{Ehrhart polynomials, $h^*$-polynomials, integer decomposition property,
 unimodality, log-concavity, Nakajima polytopes, regular unimodular triangulations.}
\subjclass[2020]{Primary: 52B20; Secondary: 05A20, 13F55, 52B11.}
\allowdisplaybreaks

\begin{document}
\raggedbottom
\begin{abstract} 
We construct counterexamples to a four-decade-old conjecture attributed to Stanley concerning the unimodality of $h^*$-polynomials of IDP polytopes. As additional applications of our main constructions, we also disprove a conjecture by Brenti on the log-concavity of $h^*$-polynomials of Gorenstein IDP polytopes, and a conjecture by Ferroni and Higashitani concerning the log-concavity of the Ehrhart series of IDP polytopes. We also answer their question about the existence of very ample polytopes with non-log-concave interior Ehrhart series. Our class of examples arises by taking Cayley sums of rectangular prisms, and hence they possess regular unimodular flag triangulations by a result of Haase, Paffenholz, Piechnik and Santos. Remarkably, for the unimodality conjecture we can even find smooth counterexamples.
\end{abstract}

\maketitle

\section{Introduction}\label{sec:introduction}

A lattice polytope $P$ is said to have the \emph{integer decomposition property} (IDP) if for every positive integer $k$, each lattice point $p\in kP\cap\mathbb{Z}^n$ can be written as a sum $p = p_1+\cdots+p_k$ where $p_1,\ldots,p_k\in P\cap \mathbb{Z}^n$. If this holds instead only for sufficiently large $k$, the polytope $P$ is said to be \emph{very ample}. The notions of IDP and very ample polytopes are of crucial relevance within the theory of toric varieties (see \cite[Chapter~1]{cox-little-schenck}).

The Ehrhart polynomial of a lattice polytope $P\subseteq \mathbb{R}^n$ is defined as the polynomial $E_P(t) \in \mathbb{Q}[t]$ such that $E_P(m)$ equals the number of lattice points in the $m$-th dilation of $P$, for every $m\in \mathbb{Z}_{\geq 1}$. The reason $E_P(t)$ is a well-defined polynomial is a remarkable result by Ehrhart \cite{ehrhart}. For a basic treatment on Ehrhart theory we suggest Beck and Robins \cite{beck-robins}.

For every polynomial $p(t)\in \mathbb{R}[t]$ of degree $d$, one can find another polynomial $h(z)\in \mathbb{R}[z]$ of degree at most $d$ such that
    \[ \sum_{m\geq 0} p(m)\, z^m = \frac{h(z)}{(1-z)^{d+1}}.\]
When $p(t)$ is the Ehrhart polynomial of a lattice polytope $P$, the numerator of the right-hand side in the above display is often denoted by $h^*_P(z)$ and termed the \emph{$h^*$-polynomial of $P$}.

The following conjecture is often attributed to Stanley, and is a weakening of a special case of \cite[Conjecture~4]{stanley-conj}. It was explicitly posed as a question also by Schepers and van Langenhoven in \cite[Question~1.1]{schepers-vanlangenhoven}.

\begin{conjecture}\label{conj:stanley}
    If $P$ is an IDP polytope, then the $h^*$-polynomial of $P$ is unimodal.
\end{conjecture}

Recall that a polynomial $h(z) = h_0 + h_1z + \cdots + h_sz^s$ is said to be \emph{unimodal} if $h_0 \leq \cdots \leq h_{j-1}\leq h_j \geq h_{j+1} \geq \cdots \geq h_s$ for some index $j$.
The history behind this conjecture is quite remarkable. There have been multiple proposed generalizations and special cases throughout the last four decades, by different authors from all kinds of mathematical backgrounds, including pure combinatorics, algebraic geometry, and commutative algebra. We refer to \cite{braun,ferroni-higashitani} for a detailed survey concerning the above conjecture and several of its variations; to name a few: Mustata and Payne \cite{mustata-payne} and Payne \cite{payne} showed the existence of non IDP reflexive polytopes whose $h^*$-polynomials are not unimodal, settling in the negative a conjecture previously posed by Hibi \cite{hibi}.

The main result of this article is the following.

\begin{theorem}\label{thm:main-unimodality}
    There exist infinitely many IDP lattice polytopes for which the $h^*$-polynomial is not unimodal. 
\end{theorem}

Furthermore, one can take these polytopes to possess regular unimodular flag triangulations. A well-known property is that the existence of a unimodular triangulation automatically implies the IDP (see \cite{deloera-rambau-santos,haase-paffenholz-piechnik-santos} for more on unimodular triangulations). To the author's taste, that they can be taken to be not just unimodular and regular but also flag is somewhat surprising, as the flagness condition is known to imply heavy restrictions on the shape of the $h^*$-polynomial (see \cite[Theorem~3.14]{ferroni-higashitani}).

In very recent years, two breakthroughs have been obtained concerning Conjecture~\ref{conj:stanley}. On one hand, Adiprasito, Papadakis, Petrotou, and Steinmeyer \cite{adiprasito-papadakis-petrotou-steinmeyer} and Adiprasito, Papadakis, and Petrotou \cite{adiprasito-papadakis-petrotou} proved that Conjecture~\ref{conj:stanley} is indeed true under the additional assumption of Gorensteinness. On the other hand, Hofscheier, Kurylenko, and Nill \cite{hofscheier-kurylenko-nill1} showed that Conjecture~\ref{conj:stanley} becomes actually false if one asks about log-concavity instead. Recall that a polynomial $h(z) = h_0 + h_1z+\cdots+h_sz^s$ with nonnegative coefficients is said to be \emph{log-concave} if $h_i^2\geq h_{i-1}h_{i+1}$ for each index $0< i < s$; if the sequence $(h_0,\ldots,h_s)$ has no internal zeros, then log-concavity is in fact stronger than unimodality. A few weeks before this manuscript was released, the same authors constructed an example of a very ample polytope for which the $h^*$-polynomial is not unimodal \cite{hofscheier-kurylenko-nill2}.

In previous years, there was a good deal of activity in this area, and some additional partial progress was achieved towards Conjecture~\ref{conj:stanley}. Notably, Athanasiadis \cite[Theorem~1.3]{athanasiadis} proved that $d$-dimensional lattice polytopes admitting a regular unimodular triangulation have an $h^*$-polynomial whose second half is decreasing, i.e., $h^*_d \leq h^*_{d-1} \leq \cdots \leq h^*_{\lfloor (d+1)/2\rfloor}$. The same result, but only under the IDP assumption was proved more recently in \cite[Corollary~2.2]{adiprasito-papadakis-petrotou}.

Our polytopes are constructed as follows. Fix $n$ and $d$, and consider any nonnegative integer matrix $A\in \mathbb{Z}_{\geq 0}^{n\times(d+1)}$. Let us label the rows with numbers from $1$ to $n$ and the columns from $0$ to $d$. Assume that none of the $n$ rows of $A$ is zero. Each column of $A$ gives rise to a rectangular prism; for each $0\leq j\leq d$ consider $D_j$ to be the rectangular prism given by $D_j = \prod_{i=1}^n [0,a_{ij}]$. Since we allow some entries of $A$ to be zero, the dimension of $D_j$ can drop.

We define the polytope $\mathcal{C}(A)$ to be the Cayley sum of the rectangular prisms $D_0,\ldots, D_d$. Recall that the \emph{Cayley sum} of a collection of polytopes $P_0,\ldots, P_d\subseteq \mathbb{R}^n$ is the convex hull of $(P_0\times \{(0,\ldots,0)\})\cup (P_1\times \{e_1\}) \cup\cdots\cup (P_d\times \{e_d\})$ in $\mathbb{R}^{n+d}$, where each $e_j$ stands for the $j$-th canonical vector in $\mathbb{R}^d$ (see, e.g., \cite{haase-nill-payne} for this operation in the context of Ehrhart theory). Since none of the rows of $A$ is identically zero, the dimension of $\mathcal{C}(A)$ is always $n+d$.

We prove that for each matrix $A$ as above, the polytope $\mathcal{C}(A)$ possesses a regular unimodular flag triangulation, and hence has the integer decomposition property. The crucial observation enabling this conclusion is that they are a special class of \emph{Nakajima polytopes} and hence \cite[Corollary~2.9]{haase-paffenholz-piechnik-santos} applies directly. Furthermore, we prove that whenever all the entries of $A$ are strictly positive, then $\mathcal{C}(A)$ is a smooth polytope, and we can even find smooth counterexamples to Conjecture~\ref{conj:stanley}.

To disprove Conjecture~\ref{conj:stanley} we employ the following strategy. 

\begin{enumerate}[\normalfont (i)]

    \item We first prove a general concrete formula for the Ehrhart polynomial of $\mathcal{C}(A)$. Concretely,
    \[ E_{\cC(A)}(k)=
 \sum_{\substack{b_0+\cdots+b_d=k\\b_j\in\Z_{\ge0}}}
       \prod_{i=1}^n\left(1+\sum_{j=0}^d a_{ij}b_j\right). \]

    \item With the preceding formula at hand, the problem becomes that of finding a matrix $A$ for which the $h^*$-polynomial is non-unimodal. For any two positive integers $\alpha$ and $\beta$, consider the matrix $A_{\alpha,\beta}$ given by
 \[
 A_{\alpha,\beta}=[\,\alpha\mathbf1_d\mid\beta I_d\,]
 =\begin{bmatrix}
 \alpha&\beta&0&\cdots&0\\
 \alpha&0&\beta&\cdots&0\\
 \vdots&\vdots&\vdots&\ddots&\vdots\\
 \alpha&0&0&\cdots&\beta
 \end{bmatrix}\in\Z_{\ge0}^{d\times(d+1)},
 \]
where $\mathbf1_d$ is the all-ones column vector and $I_d$ is the identity matrix.

Take $n=d=20$. The polytope $\mathcal{C}(A_{50,800})$ has dimension $40$ and its $h^*$-polynomial has coefficients:
    \begin{align*}
 h_{18}^*&=949^{20}-2879\cdot899^{19}+1\,288\,210\cdot849^{18},\\
 h_{19}^*&=899^{20}-1829\cdot849^{19},\\
 h_{20}^*&=849^{20}.
\end{align*}
These can be computed exactly and satisfy $h^*_{18} > h^*_{19} < h^*_{20}$.
\end{enumerate}

By taking pyramids over the polytope constructed above, or by making other choices of parameters, one can easily construct infinitely many other examples of IDP polytopes of arbitrarily large dimension with a non-unimodal $h^*$-polynomial.

Since our methods come in handy to test other conjectures, it is natural to ask if we can go beyond Stanley's conjecture. In fact, we actually came up with a counterexample to Conjecture~\ref{conj:stanley} by searching for a counterexample to another conjecture posed by Ferroni and Higashitani \cite[Conjecture~1.2]{ferroni-higashitani}. 

\begin{theorem}\label{thm:main-series}
    There exist infinitely many IDP lattice polytopes whose Ehrhart series is not log-concave. 
\end{theorem}

In other words, we show the existence of an IDP polytope for which the infinite sequence $E_P(0), E_P(1), E_P(2),\ldots$ is not log-concave. The underlying idea to construct these polytopes is overall the same, but it was the pursuit of these examples which led us to the solution. 

While three years ago, when we posed it, we were convinced that our conjecture was more likely to be true than Conjecture~\ref{conj:stanley}, it was actually easier to disprove. Although the author spent many hours interacting with Codex unsuccessfully in the pursuit of counterexamples to Conjecture~\ref{conj:stanley}, it was only after the efforts were focused on disproving \cite[Conjecture~1.2]{ferroni-higashitani} that we were able to come up with the right class of polytopes. 

Finally, we also disprove a conjecture by Brenti \cite[Conjecture~5.2]{brenti-update} via the following theorem.

\begin{theorem}\label{thm:main-gorenstein}
    There exist Gorenstein IDP polytopes whose $h^*$-polynomial is not log-concave.
\end{theorem}

This shows that the results by Adiprasito, Papadakis, Petrotou, and Steinmeyer~\cite{adiprasito-papadakis-petrotou-steinmeyer} and Adiprasito, Papadakis, and Petrotou~\cite{adiprasito-papadakis-petrotou} are in a sense ``the best possible'': they cannot be upgraded to log-concavity, nor can the Gorensteinness assumption be dropped if one hopes for unimodality.
The same Gorenstein examples also provide an affirmative answer to \cite[Question~5.12(b)]{ferroni-higashitani}. That is, there exists a very ample polytope (actually far nicer, Gorenstein and having regular unimodular flag triangulations) whose interior Ehrhart series is not log-concave, i.e., the infinite sequence $|E_P(-1)|$, $|E_P(-2)|$, $\ldots$ is not log-concave.

Since the Gorenstein members of our class of IDP polytopes have regular unimodular flag triangulations, it is natural to inquire if they disprove the analogue of Gal's conjecture for lattice polytopes. In other words, it is sensible to search for an example of a Gorenstein polytope having a regular unimodular flag triangulation and whose $h^*$-polynomial fails to be $\gamma$-positive (see \cite[Conjecture~3.28]{ferroni-higashitani}). We have not been able to find one, and indeed it is possible to prove\footnote{The proof is not included in this paper, but can be found in \cite{ferroni-slop}.} that there are no counterexamples belonging to this class.

\subsection*{Outline}
In Section~\ref{sec:prisms} we describe the geometry of $\mathcal C(A)$,
including its vertices, integer decomposition property, and regular
unimodular flag triangulations. We also give sufficient conditions for
smoothness and Gorensteinness. Section~\ref{sec:enumeration} establishes
the counting formulas for the Ehrhart polynomial, its numerator, and
interior lattice points. We then specialize the matrix $A$: in
Section~\ref{sec:nonunimodality} to obtain nonunimodal $h^*$-polynomials,
in Section~\ref{sec:series} to disprove log-concavity of Ehrhart series,
and in Section~\ref{sec:interior} to obtain the Gorenstein and interior-series
counterexamples.

\section{Cayley sums of rectangular prisms}\label{sec:prisms}

Fix integers $n\ge1$ and $d\ge0$, and let
$A=(a_{ij})\in\Z_{\ge0}^{n\times(d+1)}$ have no zero row.
For $0\le j\le d$, put
\[
 D_j=\prod_{i=1}^n[0,a_{ij}],\qquad
 v_0=0,\quad v_j=e_j\ (j\ge1).
\]
As in the introduction, define $\cC(A)$ to be the Cayley sum of the
rectangular prisms $D_0,\ldots,D_d$:
\[
 \cC(A)=\conv\left(\bigcup_{j=0}^d D_j\times\{v_j\}\right)
 \subseteq\R^{n+d}.
\]
A zero entry in $A$ allows some of these prisms to drop dimension.

Let $\Delta_d=\conv(v_0,\ldots,v_d)$, with barycentric coordinates
$\lambda_0(y)=1-\sum_{j=1}^d y_j$ and $\lambda_j(y)=y_j$ for
$1\le j\le d$. We first record the dimension and vertices of
$\cC(A)$, together with the inequality description that will be used
in the proofs below.

\begin{proposition}\label{prop:prisms}
The polytope $\cC(A)$ is an $(n+d)$-dimensional lattice polytope, and
\begin{equation}\label{eq:general-prism}
 \cC(A)=\left\{(x,y)\in\R^n\times\Delta_d:
 0\le x_i\le\sum_{j=0}^d a_{ij}\lambda_j(y)
 \quad(1\le i\le n)\right\}.
\end{equation}
Its vertices are precisely
\[
 \bigcup_{j=0}^d
 \left\{(x,v_j):x_i\in\{0,a_{ij}\}\text{ for }1\le i\le n\right\},
\]
where repeated coordinate choices are counted only once. In particular,
if $s_j=\#\{i:a_{ij}>0\}$, then the number of vertices of $\mathcal{C}(A)$ is $\sum_{j=0}^d 2^{s_j}$.
\end{proposition}
\begin{proof}
The Cayley sum is the convex hull of a set of lattice points, so it is a lattice
polytope. For fixed $y\in\Delta_d$, its fiber under projection onto
the second factor is the Minkowski sum $\sum_j\lambda_j(y)D_j$.
Since the $D_j$ are coordinate prisms, this fiber is
\[
 \prod_{i=1}^n\left[0,\sum_{j=0}^d a_{ij}\lambda_j(y)\right],
\]
which proves \eqref{eq:general-prism}. When $y$ lies in the interior
of $\Delta_d$, all fiber lengths are positive, since no row is zero.
Thus the dimension is $n+d$.

Every layer $D_j\times\{v_j\}$ is a face: it is the inverse image of
a vertex of the base simplex under projection. Its vertices are
therefore vertices of $\cC(A)$. Conversely, the Cayley sum is generated
by these vertices, so there are no others. The $j$th prism has $2^{s_j}$
vertices, and distinct layers are disjoint. This proves the count of vertices.
\end{proof}

\begin{proposition}\label{prop:triangulations}
Every polytope $\cC(A)$ admits a regular unimodular flag triangulation.
In particular, it has the integer decomposition property.
\end{proposition}
\begin{proof}
Following \cite{haase-paffenholz-piechnik-santos}, an \emph{integral chimney} over a lattice polytope $P$ is
\[
 \operatorname{Chim}(P,0,f)=\{(z,t):z\in P,\ 0\le t\le f(z)\},
\]
where $f$ is an integral affine function nonnegative on $P$.
Thanks to a result by Haase, Paffenholz, Piechnik, and Santos \cite[Theorem~2.8]{haase-paffenholz-piechnik-santos}, if $P$ has a regular unimodular flag triangulation, then so does $\operatorname{Chim}(P,0,f)$.

Start with $\Delta_d$, whose trivial triangulation is regular,
unimodular, and flag. Adjoin $x_1,\ldots,x_n$ successively, using
the upper functions $f_i(y)=\sum_j a_{ij}\lambda_j(y)$.
They are integral affine and nonnegative throughout the preceding
polytope, so the theorem applies at every step. Vanishing fiber lengths
on boundary faces are allowed. Equivalently, $\cC(A)$ is a \emph{Nakajima
polytope} since $\Delta_d$ itself is obtained from a point by successive
chimneys; see also
\cite[Corollary~2.9]{haase-paffenholz-piechnik-santos}.
The integer decomposition property follows from the existence of a
unimodular triangulation.
\end{proof}

A full-dimensional lattice polytope is \emph{smooth} if it is simple
and the primitive edge directions at every vertex form a lattice basis.

\begin{proposition}\label{prop:smooth}
If every entry of $A$ is positive, then $\cC(A)$ is smooth.
\end{proposition}
\begin{proof}
At a vertex $(x,v_j)$, precisely $d$ base inequalities and one of the
two bounds for each fiber coordinate are tight. Positivity ensures
that the two bounds never coincide. These $n+d$ inequalities have
linearly independent normals, so the vertex is simple. The $n$ edges
in its layer have primitive directions $(\pm e_i,0)$. Along each of
the other $d$ edges, the chosen lower or upper fiber bound stays tight,
while the base moves from $v_j$ towards another simplex vertex $v_\ell$.
The corresponding edge direction has integral fiber coordinates and
base component $v_\ell-v_j$, hence is primitive. The $d$ vectors
$v_\ell-v_j$ form a basis of $\Z^d$. Together with the fiber directions,
these give a block triangular matrix with unimodular diagonal blocks,
and therefore a basis of $\Z^{n+d}$.
\end{proof}

Recall that the \emph{codegree} of a lattice polytope is the smallest
positive integer $k$ for which $kP$ contains an interior lattice point.
By Ehrhart reciprocity, $|E_P(-k)|$ is the number of interior
lattice points of $kP$ for every positive integer $k$.
A polytope is \emph{Gorenstein of index $k$} if a lattice translate
of $kP$ is reflexive. Let us denote by $P^{\operatorname{relint}}$ the relative interior of a polytope $P$.

\begin{proposition}\label{prop:row-two}
Suppose every row of $A$ sums to two. Then $\cC(A)$ is Gorenstein
of index $d+1$. For every $k\ge0$,
\begin{equation}\label{eq:general-translation}
 ((k+d+1)\cC(A)^{\operatorname{relint}})\cap\Z^{n+d}
 =\mathbf1+\bigl(k\cC(A)\cap\Z^{n+d}\bigr).
\end{equation}
In particular, $|E_{\cC(A)}(-k-d-1)|=E_{\cC(A)}(k)$, by Ehrhart reciprocity.
\end{proposition}
\begin{proof}
An interior lattice point is specified by positive integers
$b_0,\ldots,b_d$ whose sum is $k+d+1$, together with fiber coordinates
satisfying
\[
 1\le x_i\le\sum_j a_{ij}b_j-1.
\]
Subtract one from every $b_j$ and every $x_i$. Since the rows sum to
two, the resulting inequalities are precisely those of $k\cC(A)$.
This gives \eqref{eq:general-translation}. There are no interior
lattice points before dilation $d+1$, since the $d+1$ positive integral
barycentric coordinates must sum to the dilation parameter. At dilation
$d+1$, every barycentric and fiber coordinate equals one.

After translating $(d+1)\cC(A)$ by this point, its fiber inequalities
have constant term one and primitive integral normals: their
coefficients on the corresponding fiber coordinate are $\pm1$.
The nonredundant base inequalities also have primitive integral normals
and constant term one. Thus every facet is at lattice distance one
from the origin, which proves reflexivity.
\end{proof}

\section{An Ehrhart polynomial computation}\label{sec:enumeration}

We now derive formulas that apply to every matrix $A$ in
\Cref{sec:prisms}. Let
\[
 F_A(u_0,\ldots,u_d)=\prod_{i=1}^n
            \left(1+\sum_{j=0}^d a_{ij}u_j\right).
\]
For a multi-index $\nu\in\Z_{\ge0}^{d+1}$, put
$|\nu|=\sum_j\nu_j$ and
$\binom{u}{\nu}=\prod_j\binom{u_j}{\nu_j}$.
Expand $F_A$ in the multivariate binomial basis:
\begin{equation}\label{eq:newton-coefficients}
 F_A(u)=\sum_{|\nu|\le n}c_\nu(A)\binom{u}{\nu},
 \qquad
 c_i(A)=\sum_{|\nu|=i}c_\nu(A)\quad(0\le i\le n).
\end{equation}
These coefficients can be obtained by forward differences. Explicitly,
\[
 c_\nu(A)=\sum_{0\le\mu\le\nu}
       (-1)^{|\nu|-|\mu|}\binom{\nu}{\mu}F_A(\mu).
\]

\begin{theorem}\label{thm:general-ehrhart}
For every nonnegative integer $k$,
\begin{equation}\label{eq:general-count}
 E_{\cC(A)}(k)=
 \sum_{\substack{b_0+\cdots+b_d=k\\b_j\in\Z_{\ge0}}}
       \prod_{i=1}^n\left(1+\sum_{j=0}^d a_{ij}b_j\right).
\end{equation}
Its $h^*$-polynomial is
\begin{equation}\label{eq:general-numerator}
 h^*_{\cC(A)}(z)=\sum_{i=0}^n c_i(A)z^i(1-z)^{n-i}.
\end{equation}
\end{theorem}
\begin{proof}
Fix the integer base coordinates $b_j=y_j$ for $j\ge1$, and put
$b_0=k-\sum_j y_j$. For a fixed such vector, the $i$th fiber coordinate
has $1+\sum_j a_{ij}b_j$ choices. Their independence proves
\eqref{eq:general-count}.

Let us substitute \eqref{eq:newton-coefficients} into this sum and use that
\[
 \sum_{\substack{b_0+\cdots+b_d=k\\b_j\ge0}}
             \prod_j\binom{b_j}{\nu_j}
 =\binom{k+d}{d+|\nu|}.
\]
Indeed, the generating function of the left-hand side is
$z^{|\nu|}/(1-z)^{d+1+|\nu|}$. Thus, after grouping terms with
equal $|\nu|$, the Ehrhart series becomes
\[
 \sum_{k\ge0}E_{\cC(A)}(k)z^k
 =\sum_{i=0}^n c_i(A)\frac{z^i}{(1-z)^{d+i+1}}.
\]
Multiplying by $(1-z)^{n+d+1}$ proves
\eqref{eq:general-numerator}.
\end{proof}

We will also use the corresponding interior count. For $k\ge1$,
strict inequalities in the fibers give
\begin{equation}\label{eq:general-interior-count}
 |E_{\cC(A)}(-k)|=
 \sum_{\substack{b_0+\cdots+b_d=k\\b_j\in\Z_{\ge1}}}
       \prod_{i=1}^n\left(\sum_{j=0}^d a_{ij}b_j-1\right).
\end{equation}
Every factor is nonnegative because $A$ has no zero row. Empty sums
are interpreted as zero. Ehrhart reciprocity takes the form
\begin{equation}\label{eq:general-reciprocity}
 \sum_{k\ge1}|E_{\cC(A)}(-k)|z^k
 =\frac{z^{n+d+1}h^*_{\cC(A)}(z^{-1})}{(1-z)^{n+d+1}}.
\end{equation}

\section{\texorpdfstring{Nonunimodal $h^*$-polynomials}{Nonunimodal h*-polynomials}}\label{sec:nonunimodality}

Set $n=d$ and consider the matrix $A_{\alpha,\beta}$ introduced in
Section~\ref{sec:introduction}. Using the concrete counting formula
of Section~\ref{sec:enumeration}, we searched by brute force over
$d,\alpha,\beta$ for a nonunimodal $h^*$-polynomial. The choice
$d=20$, $\alpha=50$, and $\beta=800$ gives the desired counterexample.
We verify it below by an exact computation.

\begin{restate}{thm:main-unimodality}
There exist infinitely many IDP lattice polytopes for which the $h^*$-polynomial is not unimodal.
\end{restate}

\begin{proof}
Take $d=20$, $\alpha=50$, and $\beta=800$.
By \Cref{prop:prisms,prop:triangulations}, the polytope
$P=\cC(A_{50,800})$ has dimension $40$ and admits a regular
unimodular flag triangulation. By \eqref{eq:general-count},
\[
 E_P(k)=\sum_{\substack{b_0+\cdots+b_{20}=k\\b_j\ge0}}
              \prod_{i=1}^{20}(50b_0+800b_i+1).
\]
Since $h_P^*(z)=(1-z)^{41}\sum_{k\ge0}E_P(k)z^k$, we have
\begin{equation}\label{eq:coefficient-extraction}
 h_r^*=\sum_{k=0}^r(-1)^{r-k}\binom{41}{r-k}E_P(k).
\end{equation}
Using a computer to evaluate these finite sums for $r=18,19,20$ gives
\begin{align*}
 h_{18}^*&=949^{20}-2879\cdot899^{19}+1\,288\,210\cdot849^{18},\\
 h_{19}^*&=899^{20}-1829\cdot849^{19},\\
 h_{20}^*&=849^{20}.
\end{align*}
For a compact exact certificate of the strict valley, these integers satisfy
\[
 100(h_{18}^*-h_{19}^*)>h_{20}^*>0,
 \qquad 100(h_{20}^*-h_{19}^*)>h_{20}^*.
\]
In particular, $h_{18}^*>h_{19}^*<h_{20}^*$.
\end{proof}

\begin{theorem}\label{thm:smooth-unimodality}
There are smooth polytopes admitting regular unimodular flag triangulations and whose $h^*$-polynomials are not unimodal.
\end{theorem}
\begin{proof}
Let $J$ be the $20\times21$ all-ones matrix and consider
$P_s=\cC(A_{50s,800s}+J)$ for positive integers $s$.
All entries are positive, so each $P_s$ is smooth by
Proposition~\ref{prop:smooth}. Here
\[
 E_{P_s}(k)=\sum_{\substack{b_0+\cdots+b_{20}=k\\b_j\ge0}}
                \prod_{i=1}^{20}(50sb_0+800sb_i+k+1).
\]
Take $s=100$. Using the formulas in \eqref{eq:coefficient-extraction}, one can use a computer to verify that the coefficients
at indices $18,19,20$ satisfy
\begin{align*}
 100\bigl(h_{18}^*(P_{100})-h_{19}^*(P_{100})\bigr)&>h_{20}^*(P_{100})>0,\\
 100\bigl(h_{20}^*(P_{100})-h_{19}^*(P_{100})\bigr)&>3h_{20}^*(P_{100}).
\end{align*}
In particular, $h_{18}^*(P_{100})>h_{19}^*(P_{100})<h_{20}^*(P_{100})$:
the coefficients decrease and then increase at index $19$, so the
$h^*$-polynomial is not unimodal.
\end{proof}

\section{Non-log-concave Ehrhart series}\label{sec:series}

For positive integers $\alpha,\beta$, define
\[
 B_{\alpha,\beta}=
 \begin{bmatrix}
 \alpha&\alpha&0&0\\
 \alpha&0&\alpha&0\\
 \alpha&0&0&\alpha\\
 0&\beta&0&0\\
 0&0&\beta&0\\
 0&0&0&\beta
 \end{bmatrix}\in\Z_{\ge0}^{6\times4}.
\]
Let $J$ now denote the $6\times4$ all-ones matrix, and put
$P_s=\cC(B_{s,s}+J)$. These are smooth $9$-dimensional polytopes,
by \Cref{prop:prisms,prop:smooth}, and possess regular unimodular
flag triangulations by \Cref{prop:triangulations}.

\begin{restate}{thm:main-series}
There exist infinitely many IDP lattice polytopes whose Ehrhart series is not log-concave.
\end{restate}

\begin{proof}
The counting formula specializes to
\begin{equation}\label{eq:smooth-series-count}
 E_{P_s}(k)=\sum_{\substack{b_0+b_1+b_2+b_3=k\\b_i\ge0}}
 \prod_{i=1}^3\bigl(s(b_0+b_i)+k+1\bigr)\bigl(sb_i+k+1\bigr).
\end{equation}
For $k=1$, the composition $(1,0,0,0)$ contributes
$8(s+2)^3$; each of the other three compositions contributes
$16(s+2)^2$. Thus $E_{P_s}(1)$ is a polynomial in $s$ of degree
$3$ with leading coefficient $8$.

For $k=2$, at most four of the six coefficients
$b_0+b_i,b_i$ are nonzero; hence $\deg_s E_{P_s}(2)\le4$.
For $k=3$, all six coefficients are nonzero only for
$(b_0,b_1,b_2,b_3)=(0,1,1,1)$, whose contribution is $(s+4)^6$.
Consequently, $E_{P_s}(3)$ has degree $6$ and leading coefficient
$1$. It follows that
\[
 E_{P_s}(1)E_{P_s}(3)-E_{P_s}(2)^2=8s^9+O(s^8)>0
\]
for all sufficiently large integers $s$. Their volumes strictly
increase with $s$, so this gives infinitely many counterexamples,
all of which are smooth.

For concreteness, take $s=1000$. Evaluating
\eqref{eq:smooth-series-count} at the three relevant dilations gives
\begin{align*}
 E_{P_{1000}}(1)&=8\,096\,288\,256,\\
 E_{P_{1000}}(2)&=82\,112\,599\,942\,290,\\
 E_{P_{1000}}(3)&=1\,074\,277\,139\,082\,321\,920.
\end{align*}
In particular,
\[
 E_{P_{1000}}(1)E_{P_{1000}}(3)-E_{P_{1000}}(2)^2
 =1\,955\,178\,315\,558\,917\,866\,776\,927\,420>0.\qedhere
\]
\end{proof}

\section{Gorenstein polytopes and interior Ehrhart series}\label{sec:interior}

Index the coordinates of $\Z^4$ by $0,1,2,3$. For a positive integer
$m$, let $A_m$ be the $6m\times4$ matrix consisting of $m$ copies
of each of the six rows
\[
 e_0+e_i,\quad 2e_i\qquad(1\le i\le3),
\]
and put $R_m=\cC(A_m)$. All rows sum to two. Therefore
\Cref{prop:prisms,prop:triangulations,prop:row-two} show that $R_m$
is a $(6m+3)$-dimensional Gorenstein IDP polytope of index $4$
admitting a regular unimodular flag triangulation. Moreover,
\begin{equation}\label{eq:R-translation}
 |E_{R_m}(-k-4)|=E_{R_m}(k)\qquad(k\ge0).
\end{equation}

Writing $E_m(k)=E_{R_m}(k)$, \eqref{eq:general-count} becomes
\[
 E_m(k)=\sum_{\substack{b_0+b_1+b_2+b_3=k\\b_i\ge0}}
      \prod_{i=1}^3(b_0+b_i+1)^m(2b_i+1)^m.
\]
Grouping the compositions for $k=1,2,3$ gives
\begin{align}
 E_m(1)&=8^m+3\cdot6^m,\label{eq:R-E1}\\
 E_m(2)&=27^m+6\cdot36^m+3\cdot15^m,\label{eq:R-E2}\\
 E_m(3)&=64^m+3\cdot108^m+3\cdot80^m+3\cdot162^m\notag\\
 &\quad+3\cdot28^m+6\cdot90^m+216^m.
 \label{eq:R-E3}
\end{align}

\begin{restate}{thm:main-gorenstein}
There exist Gorenstein IDP polytopes whose $h^*$-polynomial is not log-concave.
\end{restate}
\begin{proof}
Take $R=R_{12}$, whose dimension is $75$. The preceding formulas give
\begin{align*}
 E_R(1)&=75\,249\,823\,744,\\
 E_R(2)&=28\,580\,771\,904\,240\,372\,372,\\
 E_R(3)&=11\,304\,054\,820\,081\,833\,547\,500\,261\,376.
\end{align*}
Since $h_R^*(z)=(1-z)^{76}\sum_{k\ge0}E_R(k)z^k$, its first
three nonconstant coefficients are
\begin{align*}
 h_1^*&=E_R(1)-76=75\,249\,823\,668,\\
 h_2^*&=E_R(2)-76E_R(1)+\binom{76}{2}\\
      &=28\,580\,766\,185\,253\,770\,678,\\
 h_3^*&=E_R(3)-76E_R(2)+\binom{76}{2}E_R(1)-\binom{76}{3}\\
      &=11\,304\,052\,647\,943\,383\,287\,229\,561\,204.
\end{align*}
These satisfy
\[
 h_1^*h_3^*-(h_2^*)^2
 =33\,767\,772\,755\,382\,700\,161\,244\,430\,421\,001\,196\,588>0.
\]
Thus $h_R^*(z)$ is not log-concave.
\end{proof}

Brenti's conjecture concerns the numerator of the Hilbert series of a
standard graded Gorenstein domain. The Ehrhart ring of $R$ is such a
domain: IDP makes it standard graded, and the Gorenstein property of
$R$ makes its Ehrhart ring Gorenstein. Its Hilbert numerator is precisely
$h_R^*(z)$, so the preceding example disproves
\cite[Conjecture~5.2]{brenti-update}.

\begin{corollary}\label{cor:interior}
There exists a Gorenstein polytope admitting a regular unimodular
flag triangulation whose interior Ehrhart series is not log-concave.
\end{corollary}
\begin{proof}
Use the same polytope $R=R_{12}$. By \eqref{eq:R-translation},
$|E_R(-5)|=E_R(1)$, $|E_R(-6)|=E_R(2)$, and $|E_R(-7)|=E_R(3)$.
The values above give
\begin{multline*}
 |E_R(-5)E_R(-7)|-|E_R(-6)|^2\\
 =33\,767\,610\,161\,455\,765\,017\,372\,268\,884\,710\,005\,360>0.\qedhere
\end{multline*}
\end{proof}

This answers \cite[Question~5.12(b)]{ferroni-higashitani}
affirmatively. Part~(a) asks specifically for
$|E_P(-2)|^2<|E_P(-1)||E_P(-3)|$, but the preceding example does not settle it,
since $|E_R(-1)|=|E_R(-2)|=|E_R(-3)|=0$.

\begin{proposition}\label{prop:smooth-interior}
There exists a smooth member of the class with a non-log-concave
interior Ehrhart series.
\end{proposition}
\begin{proof}
Let $J$ denote the $180\times4$ all-ones matrix and take
$S=\cC(1000A_{30}+J)$. This polytope has dimension $183$ and is
smooth by \Cref{prop:smooth}; it also has a regular unimodular flag
triangulation. For positive integers $b_0,b_1,b_2,b_3$ summing to $k$,
put
\[
 W_k(b)=\prod_{i=1}^3
       \bigl(1000(b_0+b_i)+k-1\bigr)
       \bigl(2000b_i+k-1\bigr).
\]
Formula~\eqref{eq:general-interior-count} gives the exact finite sums
\begin{equation}\label{eq:smooth-interior-certificate}
 |E_S(-k)|=\sum_{\substack{b_0+b_1+b_2+b_3=k\\b_i\ge1}}W_k(b)^{30}.
\end{equation}
For $k=5,6,7$ there are respectively $4,10,20$ summands. Evaluating
these integer sums gives
\[
 107|E_S(-6)|^2<100|E_S(-5)||E_S(-7)|<108|E_S(-6)|^2,
\]
which in particular proves $|E_S(-6)|^2<|E_S(-5)||E_S(-7)|$.
\end{proof}

\subsection*{AI declaration}

This paper benefited from substantial use of Codex and ChatGPT 5.6 Sol. More than fifty prompts across one month were needed. I fed ChatGPT with personal notes that I had written in 2023 and which did not make it to the final version of my survey with Higashitani \cite{ferroni-higashitani}. I attempted for more than two weeks to disprove Conjecture~\ref{conj:stanley} by modifying the examples found by Hofscheier, Nill, and Kurylenko \cite{hofscheier-kurylenko-nill1,hofscheier-kurylenko-nill2}. I started seeing progress only when I changed the target and started focusing on my own conjecture with Higashitani on the log-concavity of Ehrhart series for IDP polytopes. The breakthrough happened when I asked ChatGPT about Nakajima polytopes. The results of Haase, Paffenholz, Piechnik, and Santos \cite{haase-paffenholz-piechnik-santos} were crucial. After having established counterexamples to \cite[Conjecture~1.2]{ferroni-higashitani}, the ``big picture'' became clear to me. I realized that the polytope discovered by ChatGPT was a Cayley sum of rectangular prisms, and from there the remaining findings, including the counterexamples to Conjecture~\ref{conj:stanley} came in quick succession. Except for a few calculations and mild superficial corrections, which were taken care of with the help of AI, the writing process was carried out by the author. 

\subsection*{Acknowledgments}

I am thankful to my coauthor Akihiro Higashitani, with whom I wrote the survey \cite{ferroni-higashitani}. What I learned from writing that paper with him proved extremely useful for the present paper. I am also thankful to Karim Adiprasito and Igor Pak for useful conversations in July and August 2023. I also want to thank Federico Castillo for useful feedback on a preliminary version of this manuscript.
This research was conducted using the ChatGPT subscription that I pay from my own personal money.

\bibliographystyle{amsalpha}
\bibliography{bibliography}
\end{document}